\documentclass[reqno]{amsart}

\usepackage{amsmath} 
\usepackage{amsfonts}
\usepackage{amssymb}
\usepackage{amstext}
\usepackage{amsbsy}
\usepackage{amsopn}
\usepackage{amsthm}
\usepackage{amsxtra}
\usepackage{graphicx}
\usepackage{caption}
\usepackage{subcaption}
\usepackage{color}
\usepackage{hyperref}
\usepackage{enumerate}
\usepackage{amscd}

\newtheorem{theorem}{Theorem}[section]
\newtheorem{lemma}[theorem]{Lemma}
\newtheorem{corollary}[theorem]{Corollary}

\newtheorem*{conjecture*}{Conjecture}
\newtheorem*{claim*}{Claim}
\newtheorem*{theorem*}{Theorem}

\theoremstyle{remark}

\theoremstyle{definition}
\newtheorem{definition}[theorem]{Definition}

\newcommand{\A}{\mathcal{A}}

\newcommand{\Z}{\mathbb{Z}}

\newcommand{\N}{\mathbb{N}}

\newcommand{\Aut}{{\rm Aut}}

\newcommand{\rst}[1]{\ensuremath{{\mathbin\upharpoonright}%
\raise-.5ex\hbox{$#1$}}}

\begin{document}

\title[Automorphisms of a shift of stretched exponential growth]{The automorphism group of a minimal shift of stretched exponential growth}
\author{Van Cyr}
\address{Bucknell University, Lewisburg, PA 17837 USA}
\email{van.cyr@bucknell.edu}
\author{Bryna Kra}
\address{Northwestern University, Evanston, IL 60208 USA}
\email{kra@math.northwestern.edu}

\subjclass[2010]{}
\keywords{}

\thanks{The  second author was partially supported by NSF grant 1500670.}

\subjclass[2010]{37B10, 43A07, 68R15}

\begin{abstract}
The group of automorphisms of a symbolic dynamical system is countable, but 
often very large.  For example, for a mixing subshift of finite type, the automorphism group contains isomorphic copies of the free group on two generators and the direct sum of countably many copies of $\Z$.  In contrast, 
the group of automorphisms of a symbolic system of zero entropy seems to be highly 
constrained.  Our main result is that the automorphism group of any minimal subshift of stretched exponential growth with exponent $<1/2$, is amenable (as a countable discrete group).  For shifts of polynomial growth, we further show that any finitely generated, torsion free subgroup of $\Aut(X)$ is virtually nilpotent. 

\end{abstract}

\maketitle

\section{Complexity and the automorphism group}

Let $(X, \sigma)$ be a subshift over the finite alphabet $\A$,
meaning that $X\subset \A^\Z$ is closed and invariant under the left
shift $\sigma\colon \A^\Z\to\A^\Z$. 
The group of automorphisms $\Aut(X)$ of $(X,\sigma)$ is the group of homeomorphisms of $\phi\colon X\to X$ such that $\phi\circ\sigma = \sigma\circ\phi$.  A classic result of Curtis, Hedlund, and Lyndon is that $\Aut(X)$ is always countable, but a number of results have shown that $\Aut(X)$ can be quite large.  For example, for any mixing subshift of finite type, $\Aut(X)$ always contains (among others) a copy of: every finite group, the direct sum of countably many copies of $\Z$, and the free group on two generators~\cite{Hedlund2,BLR}; every countable, locally finite, residually finite group~\cite{kimroush}; 
the fundamental group of any $2$-manifold~\cite{kimroush}.  

This extremely rich subgroup structure makes the problem of deciding when two shifts have isomorphic automorphism groups challenging.  
Moreover, Kim and Roush~\cite{kimroush} showed 
that the automorphism group of any full shift is contained in the automorphism group of any other full shift (and more generally is contained in any mixing subshift of finite type), thereby dooming 
any strategy for distinguishing two such groups that relies on finding a subgroup of one that does not 
embed into the other.  Even the question of whether the automorphism groups of the full $2$-shift and the full $3$-shift are isomorphic remains a difficult open problem~\cite{BLR} (although, as they remark, the automorphism groups of the full $2$-shift and the full $4$-shift are not isomorphic).  In all of these examples, the complicated nature of $\Aut(X)$ is a manifestation of the relatively light constraints required on $x\in\A^{\Z}$ to be a member of the shift space $X$.  Another consequence of this fact is that these shifts always have positive (although arbitrarily small) entropy.  As a corollary, if $G$ is a group that embeds into the automorphism group of the full $2$-shift, then for any $h>0$ there is a subshift of topological entropy less than $h$ into whose automorphism group $G$ also embeds. 

It is therefore natural to ask whether the automorphism group of a zero entropy subshift is more highly constrained than its positive entropy relatives.  Over the past several years, the authors~\cite{CK3, CK4} and others (e.g.,~\cite{CQY,DDMP,S,SaTo}) have shown that the zero entropy case is indeed significantly more constrained, and we continue this theme in the present work.  Specifically, we study how the the growth rate of the factor complexity $P_X(n)$, the number of nonempty cylinder sets of length $n$, constrains the algebraic properties of group of automorphisms.  
For a shift whose factor complexity grows at most linearly, we showed~\cite{CK4} that every finitely generated subgroup of $\Aut(X)$ is
virtually $\Z^d$ for some $d$ that depends on the growth rate.  We further showed~\cite{CK3} that for a transitive shift of subquadratic growth, the quotient of $\Aut(X)$ by the subgroup generated by $\sigma$, is periodic.  Further examples of minimal shifts with polynomial complexity and highly constrained automorphism groups were constructed
by Donoso, Durand, Maass and Petite~\cite{DDMP}, and an example of a minimal shift with subquadratic growth whose automorphism group is not finitely generated was given by Salo~\cite{S}.

Our main theorem provides a strong constraint on $\Aut(X)$ for any minimal subshift of stretched exponential growth with exponent $<1/2$.  We show: 
\begin{theorem}
\label{th:stretched}
If $(X, \sigma)$ is a minimal shift such that there exists $\beta < 1/2$ satisfying 
\begin{equation}
\label{eq:stretched}
\limsup_{n\to\infty}\frac{\log(P_X(n))}{n^\beta} = 0, 
\end{equation}
then $\Aut(X)$ is amenable.  Moreover, every finitely generated, torsion free subgroup of $\Aut(X)$ has subexponential growth. 
\end{theorem} 
\noindent For minimal shifts of polynomial growth, we show more: 
\begin{theorem}
\label{th:polynomial}
If $(X, \sigma)$ is a minimal shift such that there exists $d\in\N$ satisfying
\begin{equation}\label{eq:polynomial-growth}  
\limsup_{n\to\infty}\frac{P_X(n)}{n^d}=0, 
\end{equation} 
then $\Aut(X)$ is amenable.  
Furthermore, every finitely generated, torsion free subgroup 
of $\Aut(X)$ is virtually nilpotent with polynomial growth rate at most $d-1$.  In particular, the step of the nilpotent subgroup is at most $\left\lfloor\frac{-1+\sqrt{8d-7}}{2}\right\rfloor$.
\end{theorem}
In particular, this shows that $\Z^d$ does not embed in the automorphism group of a minimal shift whose growth rate 
is $o(n^d)$.  

In both the polynomial and stretched exponential cases, the amenability of $\Aut(X)$ stands in stark contrast to the possible behavior of the automorphism group of a shift of positive entropy. 
Furthermore, for shifts of positive entropy, we are far from being able to characterize the groups that can arise as automorphism groups, but we are approaching an answer for shifts with zero entropy.

This leaves open several natural questions.  We are not able to provide examples of minimal shifts showing that 
Theorems~\ref{th:stretched}  and~\ref{th:polynomial} can not be improved, in the sense 
that the automorphism groups may be smaller.  In particular, we can not rule out the possibility 
that the automorphism groups are virtually $\Z^d$, as raised in~\cite{DDMP}.  
In fact, it suffices to show that the Heisenberg group does not embed in the automorphism group, as 
any torsion free, nonabelian nilpotent group contains an isomorphic copy of the Heisenberg group as a subgroup.  
However, for subcubic growth, the control on the step of the nilpotent group in Theorem~\ref{th:polynomial} implies: 
\begin{corollary}
If $(X,\sigma)$ is a minimal shift  such that 
$$
\limsup_{n\to\infty}\frac{P_X(n)}{n^3} = 0, 
$$
then every finitely generated, torsion free subgroup of $\Aut(X)$ is virtually abelian.  
\end{corollary}

After writing this paper, Donoso, Durand, Maass, and Petite~\cite{DDMP2} shared with us a proof that improves this corollary, showing that 
the same result holds with  a growth rate that is $o(n^5)$.  

Furthermore, we believe that the amenability of $\Aut(X)$ from Theorem~\ref{th:stretched} should hold for $1/2 \leq \beta \leq 1$ thus for all zero entropy shifts, but our methods do not cover this case.  Finally, we do not know if these results generalize to transitive or general shifts.

\section{Background and Notation} 
\subsection{Subshifts} Suppose $\A$ is a finite set with the discrete topology and let $\A^{\Z}$ denote the set of bi-infinite sequences 
$$ 
\A^{\Z}:=\{(\dots,x_{-2}, x_{-1}, x_0, x_1, x_2,\dots)\colon x_i\in\A\text{ for all }i\} 
$$ 
endowed with the product topology.  The metric 
$$ 
d\bigl((\dots,x_{-1},x_0,x_1,\dots),(\dots,y_{-1},y_0,y_1,\dots)\bigr):=2^{-\min\{|i|\colon x_i\neq y_i\}} 
$$ 
generates this topology and makes $\A^{\Z}$ into a compact metric space.  When convenient, we denote an element of $\A^{\Z}$ by $(x_i)_{i=-\infty}^{\infty}$. 

The {\em left shift} is the map $\sigma\colon\A^{\Z}\to\A^{\Z}$ defined by 
$$ 
\sigma\left((x_i)_{i=-\infty}^{\infty}\right):=(x_{i+1})_{i=-\infty}^{\infty} 
$$ 
and it is easy to verify that $\sigma$ is a homeomorphism on $\A^{\Z}$.  A {\em subshift} of $\A^{\Z}$ is a closed, $\sigma$-invariant subset $X\subset\A^{\Z}$ (endowed with the subspace topology).  

If $X$ is a subshift of $\A^{\Z}$ and $f\in\A^{\{-n+1,\dots,-1,0,1,\dots,n-1\}}$ is given, then the {\em central cylinder set} $[f]_0$ of size $n$ in $X$ determined by $f$ is defined to be the set 
$$ 
[f]_0:=\{x\in X\colon x_i=f(i)\text{ for all } -n<i<n\}. 
$$ 
A standard fact is that the collection of central cylinder sets, taken over all $f$ and $n$, forms a basis for the topology of $X$.  If $g\in\A^{\{0,1,\dots,n-1\}}$, then the {\em one-sided cylinder set} $[g]_0^+$ of size $n$ in $X$ determined by $g$ is defined to be the set 
$$ 
[g]_0^+:=\{x\in X\colon x_i=g(i)\text{ for all }0\leq i<n\}. 
$$ 

For fixed $n\in\N$, the {\em words of length $n$ in $\A^{\Z}$} is the set $\A^{\{0,1,\dots,n-1\}}$ and is denoted by 
$\mathcal{L}_n(\A^{\Z})$.  An element $w\in\mathcal{L}_n(\A^{\Z})$ is written $w=(w_0,w_1,\dots,w_{n-1})$ or, when convenient, simply as a concatenation of letters: $w=w_0w_1\dots w_{n-1}$.  
A word $w$ naturally defines a one-sided central cylinder set $[w]_0$, and in a slight abuse of notation we write 
$$
[w]_0^+=\{x\in X\colon x_j=w_j \text{ for all } 0\leq j < n\}. 
$$

If $X$ is a fixed subshift, 
then the set of {\em words of length $n$ in $X$} is the set 
$$ 
\mathcal{L}_n(X):=\{w\in\mathcal{L}_n(\A^{\Z})\colon[w]_0^+\cap X\neq\emptyset\}. 
$$ 
The set of all words $\mathcal{L}(X)$ in $X$ is given by
$$ 
\mathcal{L}(X):=\bigcup_{n=1}^{\infty}\mathcal{L}_n(X) 
$$ 
and is called the {\em language of $X$}.  The {\em complexity function} $P_X\colon\N\to\N$ is defined to be 
the map 
$$ 
P_X(n):=|\mathcal{L}_n(X)|, 
$$ 
meaning it assigns to each $n\in\N$ the number of words of length $n$ in the language of $X$. 
It follows from the Morse-Hedlund Theorem~\cite{MH} that 
if $P_X$ is not increasing, then the associated subshift is periodic.  

If $w\in\mathcal{L}_n(X)$, then we say that $w$ {\em extends uniquely to its right} if there is a unique $u\in\mathcal{L}_{n+1}(X)$ such that $u_i=w_i$ for all $0\leq i<n$.  Similarly, we say that $w$ {\em extends uniquely to its left} if there is a unique $u\in\mathcal{L}_{n+1}(X)$ such that $u_i=w_i$ for all $0<i\leq n$.  
More generally, for $N\in\N$, a word $w\in\mathcal{L}_n(X)$ {\em extends uniquely $N$ times to its right} if there is a unique $u\in\mathcal{L}_{n+N}(X)$ such that $u_i=w_i$ for all $0\leq i<n$, and similarly for extensions to the left.

\subsection{Automorphisms} If $(X,\sigma)$ is a subshift, then the group of {\em automorphisms of $(X,\sigma)$} 
is the set of all homeomorphisms of $X$ that commute with $\sigma$ and this group is denoted $\Aut(X)$.  
With respect to the compact open topology, $\Aut(X)$ is discrete.  
A map $\varphi\colon X\to X$ is a {\em block code of range $R$} if for all $x\in X$ the symbol that $\varphi(x)$ assigns to $0$ is determined by the word $(x_{-R},\dots,x_0,\dots,x_R)$.  The basic structure of an automorphism of $(X,\sigma)$ is 
given by: 
\begin{theorem}[Curtis-Hedlund-Lyndon Theorem~\cite{Hedlund2}] 
If $(X,\sigma)$ is a subshift and $\varphi\in\Aut(X)$, then there exists $R\in\N\cup\{0\}$ such that $\varphi$ is a block code of range $R$. 
\end{theorem} 
If $\varphi$ is a block code of range $R$, then it is also 
a block code of range $S$ for any $S\geq R$.  Therefore we can speak of {\em a} range for $\varphi$ or of {\em the} minimal range for $\varphi$.  We define 
$$ 
\Aut_R(X):=\left\{\varphi\in\Aut(X)\colon\text{$R$ is a range for $\varphi$ and $\varphi^{-1}$}\right\}. 
$$ 
The Curtis-Hedlund-Lyndon Theorem implies that
$$ 
\Aut(X)=\bigcup_{R=0}^{\infty}\Aut_R(X). 
$$ 

We observe that it follows immediately from the definitions that if $\varphi\in\Aut_R(X)$ for some $R\in\N$, $x\in X$, and $m<n$ are integers, then the restriction of $\varphi(x)$ to $\{m,m+1,\dots,n\}$ is uniquely determined by the restriction of $x$ to $\{m-R,m-R+1,\dots,n+R-1,n+R\}$. 
This motivates the definition of the action of an automorphism on a word: 
\begin{definition} 
Suppose $\varphi\in\Aut_R(X)$ and $w\in\mathcal{L}(X)$ has length at least $2R+1$.  We 
define the word $\varphi(w)$ to be the unique $u\in\mathcal{L}_{|w|-2R}(X)$ such that for all $x\in[w]_0^+$, 
we have $\varphi(x)\in\sigma^{-R}[u]_0^+$.  
\end{definition} 

Note that $\varphi(w)$ depends on the choice of a range for $\varphi$, and so this definition  only makes sense for the pair $(\varphi,R)$ and not simply for $\varphi$. 

We record another useful consequence of these observations: 
\begin{lemma}\label{lemma:growth2} 
Let $\varphi_1,\dots,\varphi_m\in\Aut_R(X)$, let $i_1,\dots,i_n\in\{1,\dots,m\}$, and let $e_1,\dots,e_m\in\{-1,1\}$.  Then 
$$ 
\varphi_{i_1}^{e_1}\circ\varphi_{i_2}^{e_2}\circ\cdots\circ\varphi_{i_n}^{e_n}\in\Aut_{nR}(X). 
$$ 
\end{lemma} 

\subsection{Growth in groups and amenability} 
We recall some standard definitions and results; see, for example~\cite{harpe} for background.  
Suppose $G$ is a finitely generated group and $\mathcal{S}=\{g_1,\dots,g_m\}\subset G$ is a finite 
{\em symmetric generating set} for $G$, meaning that every element of $G$ can be written as a finite product of elements of $\mathcal{S}$ and that $\mathcal{S}$ is closed under inverses.  Then the {\em growth function} of the pair $(G,\mathcal{S})$ is the function 
$$ 
\gamma_G^{\mathcal{S}}(n):=\vert\{g_{i_1}g_{i_2}\cdots g_{i_k}\colon k\leq n, i_1,\dots,i_n\in\{1,\dots,m\}\}\vert,  
$$ 
which counts the number of distinct group elements that can be written as a product of at most $n$ elements of $\mathcal{S}$.  
If $x(n), y(n)\colon \N\to\N$ are nondecreasing functions, we write that $x\prec y$ if there constants $\lambda\geq 1$ and $C\geq 0$ such that $x(n)\leq \lambda y(\lambda n+C)+C$ and we write that $x\sim y$ if $x\prec y$ and $y\prec x$.  
A standard fact is that $\gamma_G^{\mathcal{S}}$ is submultiplicative, 
and that if $\mathcal{S}_1,\mathcal{S}_2$ are any two finite symmetric generating sets for $G$, 
then $\gamma_G^{\mathcal{S}_1}\sim\gamma_G^{\mathcal{S}_2}$.  
Thus we define 
$G$ to be a {\em group of exponential growth} if for some (equivalently, for every) finite symmetric generating set $\mathcal{S}$, 
$$ 
\lim_{n\to\infty}\frac{\log(\gamma_G^{\mathcal{S}}(n))}{n}>0. 
$$ 
If the limit is zero, we say that $G$ is a {\em group of subexponential growth}. 

Similarly, $G$ is a {\em group of polynomial growth rate $d$} if for some (equivalently, for every) finite symmetric generating set $\mathcal{S}$, 
$$ 
\limsup_{n\to\infty}\frac{\log(\gamma_G^{\mathcal{S}}(n))}{\log(n)}\leq d 
$$ 
and $G$ is a {\em group of weak polynomial growth rate $d$} if for some (equivalently, for every) finite symmetric generating set $\mathcal{S}$, 
$$ 
\liminf_{n\to\infty}\frac{\log(\gamma_G^{\mathcal{S}}(n))}{\log(n)}\leq d. 
$$ 
Finally, we say that $G$ is a {\em group of polynomial growth} if it is a group of polynomial growth rate $d$ for some $d\in\N$, and it is a {\em group of weak polynomial growth} if it is a group of weak polynomial growth rate $d$ for some $d\in\N$. 

A countable, discrete group $G$ is {\em amenable} if there exists a sequence $(F_k)_{k\in\N}$ of finite subsets of $G$ such that: 
\begin{enumerate} 
\item for all $g\in G$,  we have $g\in F_k$ for all but finitely many $k$; 
\item for all $g\in G$,  
$$ 
\lim_{k\to\infty}\frac{|F_k\triangle gF_k|}{|F_k|}=0. 
$$ 
\end{enumerate} 
In this case, the sequence $(F_k)_{k\in\N}$ is called a {\em F{\o}lner sequence} for $G$.

\section{Constraints on the automorphism group}

We start with a lemma which shows that if $P_X(n)$ grows slowly then there are words that extend uniquely a large number of times to both sides: 
\begin{lemma}\label{lemma:growth} 
Suppose $(X,\sigma)$ is a subshift and define 
$$ 
k_n:=\min\{k\in\N\colon\text{no word }w\in\mathcal{L}_n(X)\text{ extends uniquely $k$ times to the right and left}\}. 
$$ 
\begin{enumerate} 
\item (Polynomial growth) If there exists $d\in\N$ such that 
\begin{equation}\label{eq:eq1} 
\limsup_{n\to\infty}\frac{P_X(n)}{n^d}=0, 
\end{equation} 
then there exists $C>0$ such that for infinitely many $n\in\N$ we have $k_n>Cn$. 
\item (Stretched exponential growth) Let $0<\beta\leq1$ be fixed.  If 
\begin{equation}\label{eq:eq6} 
\limsup_{n\to\infty}\frac{\log(P_X(n))}{n^{\beta}}=0, 
\end{equation} 
then for any $C>0$ there are infinitely many $n\in\N$ such that $k_n>Cn^{1-\beta}$. 
\end{enumerate} 
\end{lemma} 
\begin{proof} 
Without loss, it suffices to assume that $P_X(n)$ 
is nondecreasing, as otherwise $(X,\sigma)$ is a periodic system.  

First assume that $(X,\sigma)$ satisfies~\eqref{eq:eq1}.  
For contradiction, suppose that 
for all $C>0$ there exists $N = N(C)\in\N$ such that $k_n\leq Cn$ for all $n\geq N$.  Pick a constant $C$ such that 
$0 < C < (2^{1/d}-1)/2$.  By the definition of $k_n$, any word of length $n$ extends in at least two ways to a word of length $n+2k_n$ (adding $k_n$ letters on each side), and any two words of length $n$ extend to different words of length $n+2k_n$ by this procedure.  So we have $P_X(n+2k_n)\geq2P_X(n)$ for all $n\in\N$.  Since $P_X$ is nondecreasing, it follows that for all $n\geq N$ we have $P_X(\lceil(1+2C)n\rceil)\geq2P_X(n)$.  By induction, it follows that for any $i\in\N$, 
\begin{equation}\label{eq:eq3} 
P_X(\lceil(1+2C)^iN\rceil)\geq2^iP_X(N).  
\end{equation} 
But by~\eqref{eq:eq1}, 
\begin{equation}\label{eq:eq4} 
P_X(\lceil(1+2C)^iN\rceil)\leq\left\lceil(1+2C)^{id}N^d\right\rceil=\left\lceil\left((1+2C)^d\right)^iN^d\right\rceil 
\end{equation} 
for all sufficiently large $i$.  By choice of $C$,  we have $(1+2C)^d<2$ and so~\eqref{eq:eq3} and~\eqref{eq:eq4} are incompatible for all sufficiently large $i$. The first statement follows.

If $(X,\sigma)$ satisfies~\eqref{eq:eq6}, define 
$$ 
d_n:=\min\{m\in\N\colon P_X(n+m)\geq2P_X(n)\} 
$$ 
to be the sequence of doubling times for $P_X$.  As in the case 
of polynomial growth, we have the trivial inequality that $d_n\leq2k_n$ for all $n\in\N$.  Again 
we proceed by contradiction and assume that 
there exists $C>0$ and $N = N(C)\in\N$ such that for all $n\geq N$ we have $k_n\leq Cn^{1-\beta}$.  Thus $d_n\leq(2C)n^{1-\beta}$ for all $n\geq N$.  If $\lambda>1$, then the doubling time for $\lambda^{n^{\beta}}$ is 
$$ 
D_n(\lambda):=\min\{m\in\N\colon\lambda^{(n+m)^{\beta}}\geq2\lambda^{n^{\beta}}\}=\left\lceil n\cdot\left(1+\log(2)/(n^{\beta}\log(\lambda))\right)^{1/\beta}-n\right\rceil. 
$$ 
Expanding the first term as a binomial series, it follows that $D_n(\lambda)$ is asymptotically 
$$ 
\frac{\log(2)}{\beta\log(\lambda)}n^{1-\beta}+o(n^{1-\beta}). 
$$ 
Thus we can find $1<\lambda<2^{1/2\beta C}$ and $M\geq N$ such that $d_n<D_n(\lambda)$ for all $n\geq M$.  
We recursively define a sequence $\{a_i\}$ by setting $a_0:=M$ and $a_{i+1}=a_i+d_{a_i}$.  If $P_X(a_i)\geq\lambda^{a_i^{\beta}-a_0^{\beta}}\cdot P_X(a_0)$, then 
$$ 
P_X(a_{i+1})\geq2P_X(a_i)\geq2\lambda^{a_i^{\beta}-a_0^{\beta}}\cdot P_X(a_0)\geq\lambda^{a_{i+1}^{\beta}-a_0^{\beta}}\cdot P_X(a_0) 
$$ 
where the last inequality holds because $a_{i+1}<a_i+D_{a_i}(\lambda)$.  By induction, $P_X(a_i)\geq\lambda^{a_i^{\beta}-a_0^{\beta}}\cdot P_X(a_0)$ for all $i$.  But then  
$$ 
\log(P_X(a_i))\geq\log(\lambda)\cdot a_i^{\beta}-\log(\lambda)\cdot a_0^{\beta}-\log(P_X(a_0)) 
$$ 
for all $i$, a contradiction of~\eqref{eq:eq6}. 
\end{proof}

\begin{lemma}\label{lemma:finite} 
Suppose $k\in\N$ and $\beta<1/2$ are fixed.  Then for all sufficiently large $N\in\N$, if $f\colon\{1,2,\dots,N\}\to\N$ is a nondecreasing function and $f(N)\leq\exp\hspace{-0.02 in}\big(N^{\beta/(1-\beta)}\big)$, then there exists an integer $M$ with $N/3\leq M\leq N-k$ and such that 
$$f(M+k)\leq f(M)\cdot\exp\hspace{-0.02 in}\big(M^{(2\beta-1)/(2-2\beta)}\big).
$$
\end{lemma} 
\begin{proof} 
Suppose not.  Then for infinitely many $N$, there exists $f_N\colon\{1,\dots,N\}\to\N$ such that for all integers $M$ satisfying 
$N/3\leq M\leq N-k$,  we have $f_N(M+k)>f_N(M)\cdot\exp(M^{\frac{2\beta-1}{2-2\beta}})$.  Let $\tilde{N}$ be the largest multiple of $3k$ less than or equal to $N$.    Then for $n\leq 2\tilde{N}/3k$, we have 
\begin{eqnarray*} 
\log\bigl(f_N(\tilde{N}/3+kn)\bigr)  && \\ 
&&\hspace{-1 in}> \log\bigl(f_N(\tilde{N}/3)\bigr)+\sum_{m=0}^{n-1}(\tilde{N}/3+km)^{\frac{2\beta-1}{2-2\beta}} \\ 
&&\hspace{-1 in}> \log\bigl(f_N(\tilde{N}/3)\bigr)+\int_0^n(\tilde{N}/3+kx)^{\frac{2\beta-1}{2-2\beta}}\,dx \\ 
&&\hspace{-1 in}= \log\bigl(f_N(\tilde{N}/3)\bigr)+\frac{2-2\beta}{k}\cdot\left((\tilde{N}/3+kn)^{\frac{1}{2-2\beta}}- 
(\tilde{N}/3)^{\frac{1}{2-2\beta}}\right). 
\end{eqnarray*} 
Taking $n=2\tilde{N}/3k$, 
\begin{eqnarray*} 
\log(f_N(\tilde{N})) & >  & \log(f_N(\tilde{N}/3))+\frac{2-2\beta}{k}\cdot\left(\tilde{N}^{\frac{1}{2-2\beta}}- (\tilde{N}/3)^{\frac{1}{2-2\beta}}\right) \\ 
& > & \frac{(2-2\beta)\cdot\tilde{N}^{\frac{1}{2-2\beta}}}{k}\cdot\left(1-\frac{1}{3^{1/(2-2\beta)}}\right). 
\end{eqnarray*} 
But 
$\frac{1}{2-2\beta}>\frac{\beta}{1-\beta}$, and so for all sufficiently large $N$, and thus sufficiently large $\tilde{N}$, this contradicts 
\begin{equation*} 
\log(f_N(\tilde{N}))\leq N^{\beta/(1-\beta)}\leq(\tilde{N}+3k)^{\beta/(1-\beta)}. \hfill\qedhere
\end{equation*} 
\end{proof}

\begin{lemma}\label{lemma:group} 
Suppose $(X,\sigma)$ is minimal and $w\in\mathcal{L}(X)$.  Then the subgroup $G_w$ of $\Aut(X)$ generated by 
$S_w$, where 
$$ 
S_w:=\left\{\varphi\in\Aut_{\lfloor|w|/2\rfloor}(X)\colon\varphi[w]_0^+\subseteq[w]_0^+\right\}, 
$$ 
is finite. 
\end{lemma} 
\begin{proof} 
Let $\varphi\in S_w$ be fixed.  
By the definition of $S_w$, if $u\in\mathcal{L}(X)$ and $x\in[wuw]_0^+$, 
then there exists $v_x\in\mathcal{L}_{|u|}(X)$ such that $\varphi(x)\in[wv_xw]_0^+$.  However, since the range of $\varphi$ is at most $\lfloor|w|/2\rfloor$, the word $v_x$ does not depend on 
choice of $x\in[wuw]_0^+$, meaning there exists $v\in\mathcal{L}_{|u|}(X)$ such that for all $x\in[wuw]_0^+$ we have $v_x=v$.  In other words, $\varphi[wuw]_0^+\subseteq[wvw]_0^+$.  
Since $(X,\sigma)$ is minimal, we have that the quantity
$$ 
K_w:=\max\{K\geq1\colon\text{ there exists } x\in[w]_0^+\text{ such that }\sigma^kx\notin[w]_0^+\text{ for all }0<k<K\} 
$$ 
is finite.  
Thus if $x\in X$, then there is a bi-infinite sequence of words of bounded length $\dots,u_{-2},u_{-1},u_0,u_1,u_2,\ldots\in\bigcup_{k=0}^{K_w}\mathcal{L}_k(X)$ such that $x$ has the form 
$$ 
\cdots wu_{-2}wu_{-1}wu_0wu_1wu_2w\cdots.
$$ 
Since $P_X$ is nondecreasing, it follows from these observations 
that $\varphi^{P_X(K_w)!}(x)=x$.  Since $\varphi^{P_X(K_w)!}$ is an automorphism of $(X,\sigma)$ and the orbit of $x$ is dense in $X$, $\varphi^{P_X(K_w)!}$ 
is the identity map.  Thus $\varphi^{-1}[w]_0^+\subseteq[w]_0^+$.  

Moreover, since $\varphi\in\Aut_{\lfloor|w|/2\rfloor}(X)$, we have $\varphi^{-1}\in\Aut_{\lfloor|w|/2\rfloor}(X)$, and so $S_w$ is closed under the operation of taking inverses.  Let $G_w$ denote the subgroup of $\Aut(X)$ generated by $S_w$.  It is immediate that if $\psi\in G_w$, 
then $\psi[w]_0^+\subseteq[w]_0^+$ (but the range of $\psi$ may not be bounded by $\lfloor|w|/2\rfloor$ and so it is not necessarily the case that $\psi\in S_w$).  As before, if $u\in\mathcal{L}(X)$ and $x\in[wuw]_0^+$, then there exists $v_x\in\mathcal{L}_{|u|}(X)$ such that $\varphi(x)\in[wv_xw]_0^+$.  We claim that the word $v_x$ is 
independent of $x$  for $x\in [wuw]_0^+$.  Writing $\psi$ as a product of elements of $S_w$, 
$$ 
\psi=\varphi_k\circ\varphi_{k-1}\circ\cdots\circ\varphi_1 
$$ 
where $\varphi_1,\dots,\varphi_k\in S_w$, there exist words $v_1,\dots,v_k\in\mathcal{L}_{|u|}(X)$ such that 
$$ 
\varphi_i[wv_{i-1}w]_0^+\subseteq[wv_iw]_0^+ 
$$ 
(here we take $v_0:=u$).  Thus $\psi[wuw]_0^+\subseteq[wv_kw]_0^+$.  Therefore, an element of $G_w$ is determined once we know the image of every cylinder set of the form $[wuw]_0^+$ with $|u|\leq K_w$.  There are only finitely many such sets and only finitely many possible images for each, and so $G_w$ is finite. 
\end{proof}

\begin{theorem}\label{thm:amenable} 
Suppose $(X,\sigma)$ is a minimal subshift and there exists $\beta<1/2$ such that 
\begin{equation}\label{eq:eq7} 
\limsup_{n\to\infty}\frac{\log(P_X(n))}{n^{\beta}}=0. 
\end{equation} 
Then $\Aut(X)$ is amenable. 
\end{theorem} 
\begin{proof} 
We show the amenability of $\Aut(X)$ by constructing a F{\o}lner sequence.  

\subsection*{Step 1 (Extendable Words)}  
By Lemma~\ref{lemma:growth}, there are infinitely many integers $n$ for which there exists a word of length $n$ that extends uniquely at least $n^{1-\beta}$ many times to both the right and the left.  Let $R\in\N$ be fixed and find $n>(4R)^{1/(1-\beta)}$ such that there exists a word $w\in\mathcal{L}_n(X)$ which extends uniquely at least $2R$ times to both the right and to the left.  
Let $\tilde{w}\in\mathcal{L}_{n+4R}(X)$ be the word obtained by extending $w$ exactly $2R$ times to both the right and 
to the left.  Then if $\varphi_1,\varphi_2\in\Aut_R(X)$ are such that $\varphi_1(\tilde{w})=\varphi_2(\tilde{w})$ (recall that $\varphi_i(\tilde{w})$ is a word of length $n+2R$), we have $\varphi_1^{-1}(\varphi_2(\tilde{w}))=w$ (a word of length $n$).  But by construction of $w$, we have $x\in\sigma^{-2R}[w]_0^+$:  
$$ 
(x_{2R},x_{2R+1},\dots,x_{2R+|w|-1})=w 
$$ 
if and only if $x\in[\tilde{w}]_0^+$.  It follows that $(\varphi_1^{-1}\circ\varphi_2)[\tilde{w}]_0^+\subseteq[\tilde{w}]_0^+$.  
Since $R\leq\lfloor|\tilde{w}|/4\rfloor$, it follows from Lemma~\ref{lemma:group} that 
$\varphi_1^{-1}\circ\varphi_2$ is an element of the set $S_{\tilde{w}}$.  If $G_{\tilde{w}}$ is the subgroup of $\Aut(X)$ generated by $S_{\tilde{w}}$, then we have shown that for $\varphi_1, \varphi_2\in\Aut_R(X)$,  
\begin{equation}\label{eq:condition} 
\text{$\varphi_1$ and $\varphi_2$ are in the same coset of $G_{\tilde{w}}$ if and only if $\varphi_1(\tilde{w})=\varphi_2(\tilde{w})$.} 
\end{equation} 
Consequently, $\Aut_R(X)$ is contained in the union of at most $P_X(|\tilde{w}|-2R)$ many cosets of $G_{\tilde{w}}$. 

\subsection*{Step 2 (Estimating Coset Growth)} We claim that for any fixed $k\in\N$, 
there exists a finite set $F_k\subset\Aut(X)$ such that $\Aut_k(X)\subset F_k$
and furthermore, if $\varphi\in\Aut_k(X)$ then 
$$ 
\frac{|F_k\triangle\varphi F_k|}{|F_k|}<2\exp\hspace{-0.02 in}\big(k^{(2\beta-1)/(2-2\beta)}\big)-2. 
$$ 
The sequence $(F_k)_{k\in\N}$ constructed in this way is then the desired F{\o}lner sequence.  

To construct the sets, fix $k\in\N$ and find $N>k$ sufficiently large that $P_X(n)\leq\exp\hspace{-0.02 in}\big(\frac{n^{\beta}}{4^{\beta/(1-\beta)}}\big)$ for all $n\geq N$.  By Lemma~\ref{lemma:growth}, there are infinitely many $n\geq N$ for which there exists a word $w\in\mathcal{L}_n(X)$ that extends uniquely at least $n^{1-\beta}$ times to both the right and the left.  By Lemma~\ref{lemma:finite}, for all sufficiently large $R$ and for any increasing $f\colon\{1,2,\dots,3R\}\to\N$, there exists $R\leq M\leq 3R-k$ such that 
\begin{equation}\label{eq:refer} 
f(M+k)\leq f(M)\cdot\exp\hspace{-0.02 in}\big(M^{(2\beta-1)/(2-2\beta)}\big). 
\end{equation}  

Fix $n\geq N$ such that there exists $w\in\mathcal{L}_n(X)$ that extends at least $n^{1-\beta}$ times to both the right 
and the left, and sufficiently large that if $R:=\lfloor n^{1-\beta}/9\rfloor$ then~\eqref{eq:refer} holds and $6R<n^{1-\beta}$.  Then $w$ extends uniquely at least $6R$ times to both the right and the left, and moreover we have $|w|\leq(9R)^{1/(1-\beta)}$.  Let $\tilde{w}$ be the word obtained by extending $w$ exactly $6R$ times to both the right and the left.  
Recall that if $\varphi\in\Aut_{R}(X)$, then $\varphi$ may also be thought of as a range $R+1$ automorphism, meaning 
there is a natural embedding $\Aut_{R}(X)\hookrightarrow\Aut_{R+1}(X)$.   
(The distinction between $\Aut_{R}(X)$ and its embedded image is simply that an element of $\Aut_{R}(X)$ takes a word of length $n$ to a word of length $n-2R$, 
whereas an element of $\Aut_{R+1}(X)$ takes a word of length $n$ to a word of length $n-2R-2$.)  
We use the notation $\Aut_{R}(X)\hookrightarrow\Aut_S(X)$ to refer to the embedded image of $\Aut_{R}(X)$ in $\Aut_S(X)$ for $S\geq R$.  

Define $f\colon\{1,2,\dots,3R\}\to\N$ to be 
$$ 
f(n)=\bigl\vert\left\{\varphi(\tilde{w})\in\mathcal{L}_{|w|+6R}(X)\colon\varphi\in\Aut_n(X)\hookrightarrow\Aut_{3R}(X)\right\}\bigr\vert. 
$$ 
From the inclusion $\Aut_i(X)\hookrightarrow\Aut_{i+1}(X)$,  it follows that $f(n)$ is nondecreasing.  Furthermore, 
\begin{eqnarray*} 
f(3R)&\leq&P_X(|w|+3R) \\ 
&\leq&P_X(\lfloor(9R)^{1/(1-\beta)}\rfloor+3R) \\ 
&\leq&P_X\hspace{-0.02 in}\left(\left\lfloor(12R)^{1/(1-\beta)}\right\rfloor\right) \\ 
&\leq&\exp\bigl((3R)^{\beta/(1-\beta)}\bigr).
\end{eqnarray*} 
By~\eqref{eq:refer}, there exists $R\leq M\leq 3R-k$ such that the inequality $f(M+k)\leq f(M)\cdot\exp\hspace{-0.02 in}\big(M^{(2\beta-1)/(2-2\beta)}\big)$ holds.  Since $f$ is nondecreasing, it follows immediately that 
\begin{equation*} 
f(M+i)\leq f(M)\cdot\exp\hspace{-0.02 in}\big(M^{(2\beta-1)/(2-2\beta)}\big)\text{ for all }1\leq i\leq k 
\end{equation*} 
(note that we only make use of this for $i=k$).  
Since for all $1\leq i\leq k$, the set $\Aut_{M+i}(X)$ is covered using~\eqref{eq:condition} by at most $f(M+i)$ cosets of $G_{\tilde{w}}$, we can define $F_k$ to be the set
$$ 
F_k:=\bigcup_{i=1}^{f(M)}\varphi_i\cdot G_{\tilde{w}}, 
$$ 
where $\varphi_i\in\Aut_M(X)$ for all $i$.  It is immediate that $\Aut_M(X)\subset F_k$ and hence $\Aut_k(X)\subset F_k$. 
If $\varphi\in\Aut_k(X)$, then $\varphi\circ\varphi_i\in\Aut_{M+k}(X)$ for all $i$.  Since $\Aut_{M+k}(X)$ can be covered by $f(M+k)\leq f(M)\cdot\exp\hspace{-0.02 in}\big(M^{(2\beta-1)/(2-2\beta)}\big)$ cosets of $G_{\tilde{w}}$, it follows that at most $f(M)\cdot\left(-1+\exp\hspace{-0.02 in}\big(M^{(2\beta-1)/(2-2\beta)}\big)\right)$ additional cosets (beyond those that already appear in the definition of $F_k$) are needed to cover $\Aut_{M+k}(X)$, and hence to also cover $\varphi\cdot F_k$.  Consequently, 
$$ 
\frac{|F_k\triangle\varphi F_k|}{|F_k|}\leq2\exp\hspace{-0.02 in}\big(M^{(2\beta-1)/(2-2\beta)}\big)-2\leq2\exp\hspace{-0.02 in}\big(k^{(2\beta-1)/(2-2\beta)}\big)-2. 
$$ 
Since $\beta<1/2$, this quantity tends to $0$ as $k\to\infty$.  Therefore, $(F_k)_{k\in\N}$ is a F{\o}lner sequence and $\Aut(X)$ is amenable. 
\end{proof}

We complete the proof of Theorem~\ref{th:stretched} with: 
\begin{theorem}\label{thm:subexponential} 
Suppose $(X,\sigma)$ is a minimal subshift and there exists $\beta<1/2$ such that 
$$ 
\limsup_{n\to\infty}\frac{\log(P_X(n))}{n^{\beta}}=0. 
$$ 
Then every finitely generated and torsion free subgroup of $\Aut(X)$ has subexponential growth.  Moreover, there are infinitely many $n$ for which the growth function of any such subgroup is at most $\exp(n^{\beta/(1-\beta)})$. 
\end{theorem} 
\begin{proof} 
Let $G$ be a finitely generated, torsion free subgroup of $\Aut(X)$ and let $R\in\N$ be fixed.  We estimate 
the growth rate of $|\Aut_{R\cdot m}(X)\cap G|$ as $m\to\infty$.  
By the second part of Lemma~\ref{lemma:growth}, there are infinitely many integers $n$ for which there exists a word $w_n\in\mathcal{L}_n(X)$ that extends uniquely at least $n^{1-\beta}$ times to both the right and the left.  
For each such $n$, let $m_n\in\N$ be the largest integer for which $R\cdot m_n\leq n^{1-\beta}/2$.  
Then $w$ extends uniquely at least $2R\cdot m_n$ times to both the right and the left, 
and $|w_n|\leq(2R\cdot m_n+2R)^{1/(1-\beta)}$.  
Let $\tilde{w}_n$ be the word obtained by extending $w_n$ exactly $2R\cdot m_n$ times to both the right and the left.  Then, as in first step (Extendable Words) 
of the proof of Theorem~\ref{thm:amenable}, we have that if $\varphi_1,\varphi_2\in\Aut_{R\cdot m_n}(X)\cap G$ and $\varphi_1(\tilde{w})=\varphi_2(\tilde{w})$,  
then $\varphi_1$ and $\varphi_2$ lie in the same coset of $G_{\tilde{w}}$.  But by Lemma~\ref{lemma:group}, $G_{\tilde{w}}$ is finite and since $G$ is torsion free, it follows that $\varphi_1=\varphi_2$.  Hence the size of $\Aut_{R\cdot m_n}(X)\cap G$ is bounded by the number of words of the form $\varphi(\tilde{w}_n)$, where $\varphi\in\Aut_{R\cdot m_n}(X)\cap G$.  So, 
\begin{eqnarray*} 
\log|\Aut_{R\cdot m_n}(X)\cap G|&\leq&\log\left(P_X(|w_n|+R)\right) \\ 
&\leq&\log\left(P_X(\lceil(2R\cdot m_n+2R)^{1/(1-\beta)}\rceil+R)\right) \\ 
&=&o\hspace{-0.02 in}\left(\left(\left\lceil(2R\cdot m_n+2R)^{1/(1-\beta)}\rceil+R)\right\rceil+R\right)^{\beta}\right) \\ 
&=&o\left((m_n)^{\beta/(1-\beta)}\right). 
\end{eqnarray*} 
But $m_n\to\infty$ when $n\to\infty$, and so this condition holds infinitely often.  If $\varphi_1,\dots,\varphi_m\in\Aut(X)$ is a (symmetric) set of generators for $G$ and if $R$ is chosen such that $\varphi_1,\dots,\varphi_m\in\Aut_R(X)$, then for any $k\in\N$ and any $i_1,\dots,i_k\in\{1,\dots,m\}$, we have 
$$ 
\varphi_{i_1}\circ\varphi_{i_2}\circ\cdots\circ\varphi_{i_k}\in\Aut_{R\cdot k}(X). 
$$ 
There are infinitely many $k$ for which there exists $n$ such that $k=m_n$, 
and for all such $k$ we have that log of the number of reduced words of length at most $Rk$ that may be written in $\varphi_1,\dots,\varphi_m$ is also $o\hspace{-0.02 in}\left(k^{\beta/(1-\beta)}\right)$.  If $\gamma\colon\N\to\N$ is the growth function of $G$, then 
$$ 
\liminf_{k\to\infty}\frac{\log\gamma(Rk)}{Rk}=0. 
$$ 
But since $\gamma$ is submultiplicative, the limit converges, and so 
$$ 
\lim_{k\to\infty}\frac{\gamma(Rk)}{Rk}=0 
$$ 
and $G$ has subexponential growth.  The second claim in the theorem follows immediately from the 
estimate of $\gamma(Rm_n)$.  
\end{proof} 

We complete the proof of Theorem~\ref{th:polynomial} with: 
\begin{theorem} 
Suppose $(X,\sigma)$ is a minimal subshift and there exists $d\in\N$ such that 
$$ 
\limsup_{n\to\infty}\frac{P_X(n)}{n^d}=0. 
$$ 
Then every finitely generated and torsion free subgroup of $\Aut(X)$ is virtually nilpotent.  Moreover, the step of the nilpotent subgroup is at most $\left\lfloor\frac{-1+\sqrt{8d-7}}{2}\right\rfloor$. 
\end{theorem} 
\begin{proof} 
The proof is almost identical to that of Theorem~\ref{thm:subexponential}.  This time, by the first part of Lemma~\ref{lemma:growth}, there exists $C>0$ such that for infinitely many $n$ there is a word $w_n\in\mathcal{L}_n(X)$ that extends uniquely at least $Cn$ times to the right and the left.  As before, fix $R\in\N$ and for each such $n$ define $m_n=\lfloor Cn/2R\rfloor$.  Then $w_n$ extends uniquely at least $2R\cdot m_n$ times to both the right and the left and $|w_n|\leq 2R\cdot m_n/C+2R$. 
Then proceeding in the same way, if $G$ is a finite subgroup of $\Aut(X)$, then 
\begin{eqnarray*} 
|\Aut_{R\cdot m_n}(X)\cap G|&\leq&P_X(|w_n|+R) \\ 
&\leq&P_X(\lceil2R\cdot m_n/C+2R\rceil) \\ 
&=&o\left(\lceil2R\cdot m_n/C+2R\rceil^d\right)=o((m_n)^d).  
\end{eqnarray*} 

Proceeding as in the proof of Theorem~\ref{thm:subexponential}, we have that 
$$ 
\liminf_{k\to\infty}\frac{\log\gamma(Rk)}{(Rk)^d}=0.  
$$ 
By van den Dries and Wilkie's~\cite{VW} strengthening of Gromov's Theorem~\cite{gromov} that assumes that the complexity bound need only hold for infinitely many values, $G$ is virtually nilpotent and the nilpotent subgroup has polynomial growth rate at most $d-1$.  By the Bass-Guivarc'h Formula~\cite{bass, guivarch}, the step is at most $\left\lfloor\frac{-1+\sqrt{8d-7}}{2}\right\rfloor$. 
\end{proof}

\end{document}